\documentclass{amsart}
\newtheorem{theorem}{Theorem}[section]

\newtheorem{example}[theorem]{Example}
\newtheorem{remark}[theorem]{Remark}
\def\erre{{\rm I\!R}}
\def\R{{\rm I\!R}}

\def\meas{\mathop{\rm meas}}
\def\d{\mathop{\rm dist}}
\title[Multiple solutions of...]{Multiple solutions of $p$-biharmonic equations with Navier boundary
conditions}
\author{Giovanni Molica Bisci}
\address[G. Molica Bisci]{Dipartimento MECMAT, University of Reggio Calabria,
 Via Graziella, Feo di Vito, 89124 Reggio Calabria, Italy.} \email{gmolica@unirc.it}
\author{Du\v{s}an Repov\v{s}}
\address[D. Repov\v{s}]{Faculty of Education, and Faculty of Mathematics and Physics, University of Ljubljana, POB 2964, Ljubljana, Slovenia 1001.
}
\email{dusan.repovs@guest.arnes.si}

\thanks{Typeset by \LaTeX}
\begin{document}

\begin{abstract}
In this paper, exploiting variational methods, the existence of multiple weak solutions  for a class of elliptic Navier boundary problems involving the $p$-biharmonic operator is investigated. Moreover, a concrete example of an application is presented.
\end{abstract}
\maketitle
\section{Introduction}
Motivated also by the fact that such kind of problems are used to describe a large class of physical phenomena, many authors looked
for multiple solutions of elliptic equations involving
biharmonic and $p$-biharmonic type operators:
see, for instance, the papers
\cite{GG,LT,WS,WZ,YGY}. In the present work we are interested in the existence of multiple weak solutions for the following nonlinear elliptic
Navier boundary value problem involving the $p$-biharmonic operator
\begin{equation}\tag{$H_\lambda^{f}$} \label{N}
\left\{
\begin{array}{ll}
\Delta(|\Delta u|^{p-2}\Delta u)=\lambda f(x,u)\quad & {\rm in\ }\Omega\\
u=\Delta u=0\quad & {\rm on\ }\partial\Omega,
\end{array}
\right.
\end{equation}
where $\Omega$ is an open bounded subset of $\R^N$ with a smooth enough boundary
$\partial \Omega $, $p>\max\{1,N/2\}$,
$\Delta$ is the usual Laplace operator, $\lambda$ is a positive
parameter and $f$ is a suitable continuous function defined on the set $\bar\Omega\times\R$.\par
For $p=2$, the linear operator $\Delta^2u:=\Delta(\Delta u)$ is the iterated Laplace which multiplied with a positive
constant often occurs in Navier-Stokes equations as a viscosity coefficient. Moreover, its reciprocal
operator denoted $(\Delta^{2}u)^{-1}$
is the celebrated Green operator (see \cite{Gr}).\par
In \cite{WZ}, a Navier boundary value problem is treated where the left-hand side of the equation involves an operator that is more general than the $p$-biharmonic. Meanwhile in \cite{LM}, a concrete example of application of such mathematical model to describe a physical phenomena is also pointed out.\par
 Further, by using the abstract and technical approach developed in \cite{BoMo, BoMo2,CM}, the authors are interested in looking for the existence of infinitely many weak solutions of  perturbed $p$-biharmonic equations.\par

Here, requiring a suitable growth of the primitive of $f$, we are able to establish suitable intervals of values of the parameter $\lambda$ for which the problem (\ref{N})
admits at least three weak solutions.

More precisely, the main result
ensures the existence of two real intervals of parameters $\Lambda_1$
and $\Lambda_2$
such that, for each $\lambda\in \Lambda_1\cup\Lambda_2$, the problem
(\ref{N}) admits at least three weak solutions whose norms are
uniformly bounded with respect to every $\lambda\in \Lambda_2$ (see Theorem \ref{Main}).
\par
 Our method is mostly based on a useful critical point theorem given in \cite[Theorem 3.1]{Bonanno} (see Theorem \ref{abstract} below). 
 We also cite a recent monograph by Krist\'aly, R\u adulescu and Varga \cite{KRV} as a general reference on variational methods adopted here.\par
 
 The obtained results are related to some recent contributions from \cite[Theorem 1]{LT} where, by using a critical point result from \cite{Ric}, the existence of at least three weak solutions has been obtained (see also \cite[Theorem 1]{LT2}).
  We emphasize that, in our cases, on the contrary of the above mentioned works, we give a qualitative analysis of the real intervals $\Lambda_i$ ($i=1,2$) for which problem (\ref{N})
admits multiple weak solutions (see, for details, Remarks \ref{intervalli} and \ref{general}).\par
 As an example, we present a special case of our results (see Theorem \ref{Main2} and Remark \ref{semplice} for more details) on the existence of two nontrivial weak solutions.\par
\begin{theorem} \label{intro}
Let $p>\max\{1,N/2\}$ and $f:\R \rightarrow [0,+\infty[$ be a continuous and nonzero function.
Hence, consider the following autonomous problem
\begin{equation}\tag{$G_\lambda^{f}$} \label{5N}
\left\{
\begin{array}{ll}
\Delta(|\Delta u|^{p-2}\Delta u)=\lambda f(u)\quad & {\rm in\ }\Omega\\
u=\Delta u=0\quad & {\rm on\ }\partial\Omega.
\end{array}
\right.
\end{equation}
 Assume that there exists a real constant $\gamma>0$ such that $f(t)=0$ for every $t\in [-\gamma,\gamma]$, in addition to
$$
\displaystyle\lim_{|t|\rightarrow \infty
}\frac{f(t)}{|t|^{s-1}}=0,
$$
\noindent for some $1\leq s\leq p$.\par
 \noindent Then there exist two real intervals of parameters $\Lambda'_1$ and $\Lambda'_2$ such that$:$ for every $\lambda\in\Lambda'_1$ problem \eqref{5N} admits two distinct nontrivial weak solutions in $W^{2,p}(\Omega)\cap W_0^{1,p}(\Omega)$ and, moreover, for each $\lambda\in\Lambda'_2$ there are two distinct nontrivial weak solutions in $W^{2,p}(\Omega)\cap W_0^{1,p}(\Omega)$ uniformly bounded in norm with respect to the parameter $\lambda$.\par
\end{theorem}

For completeness, we refer the reader interested in fourth-order two-point boundary
value problems to papers \cite{GST,HZ,LL, MP} and references
therein.\par
 The plan of the paper is as follows. Section 2 is devoted to our abstract framework, while Section 3 is dedicated to the main results and their consequences in the autonomous case. A concrete example of an application is then presented (see Example \ref{esempio}).\par
\section{Preliminaries}
Here, and in the sequel, $\Omega$ is an open bounded subset of $\R^N$, $p>\max\{1,N/2\}$, while $X$ denotes a separable and reflexive real Banach space $W^{2,p}(\Omega)\cap W_0^{1,p}(\Omega)$ endowed with the norm
\begin{equation}\label{norm}
\|u\|=\left(\int_\Omega |\Delta u(x)|^p dx\right)^{1/p},\quad \forall\; u\in X.
\end{equation}
\indent The Rellich-Kondrachov theorem assures that $X$ is compactly
imbedded in $C^0(\bar\Omega)$, whenever
\begin{equation}\label{immersion}
k:=\sup_{u\in X\setminus\{0\}}\frac{\|u\|_{C^0(\bar\Omega)}}{\|u\|}<+\infty,
\end{equation}
\noindent where $\|u\|_{C^0(\bar\Omega)}:=\displaystyle\sup_{x\in \bar\Omega}|u(x)|$, for every $u\in X$.\par
\indent Moreover, if $N\geq 3$, $\partial \Omega$ is of class $C^{1,1}$ and $p\in ]N/2,+\infty[$, due to Theorem 2 and \cite[Remark 1]{Talenti}, one has the following upper bound
$$
k\leq \displaystyle\meas(\Omega)^{\frac{2}{N}+\frac{1}{p'}-1}\frac{\Gamma(1+N/2)^{2/N}}{N(N-2)\pi}\Big[\frac{\Gamma(1+p')\Gamma(N/(N-2)-p')}{\Gamma(N/(N-2))}\Big]^{1/p'},
$$
where $\Gamma$ is the Gamma function, $p'$ the conjugate exponent of $p$ and $``\meas(\Omega)"$ denotes the Lebesgue measure of $\Omega$.\par
 For our aim, the main tool is a
critical points theorem contained in
\cite[Theorem 3.1]{Bonanno} which we recall here for the reader's convenience.
\begin{theorem}\label{abstract}
 Let $X$ be a separable and reflexive real Banach space$;$ $\Phi:X\to\R$ a nonnegative, continuously G$\hat{a}$teaux differentiable and sequentially weakly lower semicontinuous functional whose G$\hat{a}$teaux derivative admits a continuous inverse on $X^*$ and $\Psi :X\to \R$ a continuously G$\hat{a}$teaux differentiable functional whose $G\hat{a}$teaux derivative is compact. Assume that there exists $u_0\in X$ such that $$ \Phi (u_0)=\Psi (u_0)=0,$$ and that
  \begin{itemize}
\item[$(\textrm{i})$] $\displaystyle\lim_{\|u\|\rightarrow \infty}(\Phi(u)-\lambda\Psi(u))=+\infty,$
 \end{itemize}
  for all $\lambda\in [0,+\infty[$. Further, assume that there are $r>0$ and $\bar{u}\in X$ such that$:$
\begin{itemize}
\item[$(\textrm{ii})$] $r<\Phi(\bar{u})$$;$
\item[$(\textrm{iii})$] $ \displaystyle\sup_{u\in \overline{\Phi^{-1}(]-\infty,r[)}^{w}}\Psi(u)< \frac{r}{r+\Phi(\bar u)}\Psi(\bar{u})$$.$
\end{itemize}
Then, for each
$$\lambda \in \Lambda_{1}:=\left]\frac{\Phi(\bar{u})}{\Psi(\bar{x})-\displaystyle\sup_{u\in \overline{\Phi^{-1}(]-\infty,r[)}^{w}}\Psi(u)},\frac{r}{\displaystyle\sup_{u\in \overline{\Phi^{-1}(]-\infty,r[)}^{w}}\Psi(u)}\right[,$$ the equation
\begin{equation}\label{equationx}
\Phi'(u)-\lambda\Psi'(u)=0,
\end{equation}
has at least three distinct solutions in $X$ and, moreover, for each $h>1$, there exists an open interval
$$
\Lambda_2\subset\left[0,\frac{hr}{\displaystyle r\frac{\Psi(\bar{u})}{\Phi(\bar{u})}-\displaystyle\sup_{u\in \overline{\Phi^{-1}(]-\infty,r[)}^{w}}\Psi(u)}\right],
$$
and a positive real number $\sigma>0$ such that, for each $\lambda\in \Lambda_2$, the equation \eqref{equationx}
has at
least three solutions in $X$ whose norms are less than $\sigma$.

\end{theorem}
Note that, in the above result, the symbol $\overline{\Phi^{-1}(]-\infty,r[)}^{w}$ denotes the weak closure of the sublevel $\Phi^{-1}(]-\infty,r[)$. For completeness, given an operator $S:X\rightarrow X^*$, we say that $S$ admits a continuous inverse on $X^*$ if there exists a continuous operator $T:X^*\rightarrow X$ such that $T(S(x))=x$ for all $x\in X$.

\begin{remark}\label{intervalli}\rm{
As observed in \cite[Remark 2.1]{Bonanno}, the real intervals $\Lambda_1$
and $\Lambda_2$
in Theorem \ref{abstract} are such that either
$$\Lambda_1\cap \Lambda_2=\emptyset,$$
or
$$\Lambda_1\cap \Lambda_2\neq\emptyset.$$
In the first case, we actually obtain two distinct open intervals of positive real
parameters for which equation \eqref{equationx}
admits two nontrivial solutions; otherwise,
we achieve only one interval of positive real parameters, precisely $\Lambda_1\cup\Lambda_2$, for which equation
\eqref{equationx} admits three solutions and, in addition, the subinterval $\Lambda_2$
for
which the solutions are uniformly bounded.}
\end{remark}
\section{Main results}

\indent Let
 $$
 \tau:=\sup_{x\in\Omega}\d(x,\partial\Omega).
 $$
 \noindent Simple calculations show that there is $x_0 \in \Omega$ such that
 $B(x_0,\tau) \subseteq \Omega$, where $B(x_0,\tau)$ denotes the open ball with center $x_0$ and radius $\tau$. Now, fix $\delta>0$ and consider the function $u_\delta\in X$ defined by
\[
u_\delta(x):= \left\{
\begin{array}{ll}
0 & \mbox{ if $x \in \bar\Omega \setminus B(x^0,\tau)$} \\\\
\displaystyle 16\frac{l^2}{\tau^4}\left(\tau- l \right)^2\delta
& \mbox{ if $x \in B(x^0,\tau) \setminus B(x^0,\tau/2)$} \\\\
\delta & \mbox{ if $x \in B(x^0,\tau/2),$}
\end{array}
\right.
\]
\noindent where $l:=\sqrt{\sum_{i=1}^{N}(x_i-x^0_i)^2}$.\par
 \noindent At this point, let
$$
F(x,\xi):=\int_0^{\xi} f(x,t)dt,\quad\forall\; (x,\xi)\in\bar\Omega\times\R,
$$
and put
$$
R_{F}(\tau,\delta):=\int_{B(x^0,\tau) \setminus B(x^0,\tau/2)}F(x,u_\delta(x))\;dx.
$$
 \noindent Moreover, set\par
$$
\sigma_{p,N}({\tau}):=\int_{\tau/2}^{\tau}|2(N+2)s^2-3(N+1)\tau s+N\tau^2|^ps^{N-1}ds.
$$
\noindent Finally, let us denote
$$
K_{p,N}(\tau):=\frac{\displaystyle\tau^{4p}\Gamma(N/2)}{2^{5p+1}\pi^{N/2}k^p\sigma_{p,N}(\tau)},
$$
and, for $\gamma>0$, define
$$
\eta(\gamma,\delta):=\frac{\tau^{4p}\Gamma(N/2)\gamma^p}{\tau^{4p}\Gamma(N/2)\gamma^p+k^p2^{5p+1}\pi^{N/2}\delta^{p}\sigma_{p,N}(\tau)}.
$$
\noindent With the above notations, the main result reads as follows.
\begin{theorem}\label{Main}
Let $f\in C^0(\bar\Omega\times\R)$ and put
$$
F(x,\xi):=\int_0^{\xi} f(x,t)dt,\quad\forall\; (x,\xi)\in\bar\Omega\times\R.
$$ Assume that there exist two positive constants $\gamma$ and $\delta$ such that
\begin{itemize}
\item [$(\textrm{h}_1)$] $\delta>K_{p,N}(\tau)^{1/p}\gamma$ $;$
\item [$(\textrm{h}_2)$] The following inequality holds
$$\displaystyle\int_{\Omega}\max_{|\xi|\leq \gamma}F(x,\xi)\;dx<\eta(\gamma,\delta)\left(R_{F}(\tau,\delta)
+\int_{B(x^0,\tau/2)}F(x,\delta)\;dx\right).$$
\end{itemize}
Further, require that
\begin{itemize}
\item [$(\textrm{h}_3)$]
There exist a function $\alpha\in L^1(\Omega)$ and a positive constant $s$ with $s<p$ such that
$$
F(x,\xi)\leq \alpha(x)(1+|\xi|^s),
$$
\noindent for almost every $x\in\Omega$ and for every $\xi\in\R$.
\end{itemize}
Then, for each
$$\lambda \in \Lambda_{1}:=\left]\lambda_1,\lambda_2\right[,$$

\noindent where
$$
\lambda_1:=\frac{2^{5p+1}\pi^{N/2}\sigma_{p,N}({\tau})\delta^p}{\tau^{4p}\Gamma(N/2)p\left(R_{F}(\tau,\delta)
+\displaystyle\int_{B(x^0,\tau/2)}F(x,\delta)\;dx-\displaystyle\int_{\Omega}\max_{|\xi|\leq \gamma}F(x,\xi)\;dx\right)},
$$
and
$$
\lambda_2:=\frac{\gamma^p}{\displaystyle pk^p\int_{\Omega}\max_{|\xi|\leq \gamma}F(x,\xi)\;dx},
$$
problem \eqref{N}
has at least three distinct solutions in $X$ and, moreover, for each $h>1$, there exists an open interval
$$
\Lambda_2\subset\left[0,\lambda_{3,h}\right],
$$
where
$$
\lambda_{3,h}:=\frac{h\gamma^p/(pk^p)}{\displaystyle\frac{\gamma^p\left(R_{F}(\tau,\delta)
+\displaystyle\int_{B(x^0,\tau/2)}F(x,\delta)\;dx\right)\tau^{4p}\Gamma(N/2)}{\displaystyle 2^{5p+1}k^p\pi^{N/2}\sigma_{p,N}(\tau)\delta^p}-\int_{\Omega}\max_{|\xi|\leq \gamma}F(x,\xi)\;dx},
$$
and a positive real number $\sigma>0$ such that, for each $\lambda\in \Lambda_2$, problem \eqref{N}
has at
least three solutions in $X$ whose norms are less than $\sigma$.
\end{theorem}
\begin{proof}
For each $u\in X$, let $\Phi,\Psi:X\to\R$ defined by setting
$$
\Phi(u):=\frac{\|u\|^p}{p},\quad\quad \Psi(u):=\int_\Omega F(x,u(x))dx.
$$
 It is easy to verify that $\Phi:X\to\R$ is a nonnegative, continuously G\^{a}teaux differentiable and sequentially weakly lower semicontinuous functional whose G\^{a}teaux derivative admits a continuous inverse on $X^*$. Meanwhile, $\Psi$ is continuously G\^{a}teaux differentiable
with compact derivative and, moreover, $\Phi (u_0)=\Psi (u_0)=0,$ where $u_0$ is the identically zero function in $X$. In particular, one has
$$
\Phi'(u)(v)=\int_\Omega |\Delta u(x)|^{p-2}\Delta u(x)\Delta v(x)dx,
$$
and
$$
\Psi'(u)(v)=\int_\Omega f(x,u(x))v(x)dx,
$$
for every $u,v\in X$.\\
\indent Now, fixing $\lambda>0$, if we recall that a weak solution of problem (\ref{N}) is a function $u\in X$ such that
\[
\int_\Omega |\Delta u(x)|^{p-2}\Delta u(x)\Delta v(x)dx=\lambda\int_\Omega f(x,u(x))v(x)dx,
\]
for every $v\in X$,
it is obvious that our goal is to find critical points of the energy
functional $J_\lambda:=\Phi-\lambda\Psi$.\par
\indent Thanks to hypothesis $(\textrm{h}_3)$ and bearing in mind (\ref{immersion}), one has
$$
\int_\Omega F(x,u(x))dx\leq \|\alpha\|_{L^1(\Omega)}(1+k^s\|u\|^{s}).
$$
\noindent Hence
$$
J_\lambda(u)\geq \frac{\|u\|^p}{p}-\lambda \|\alpha\|_{L^1(\Omega)}(1+k^s\|u\|^{s}).
$$
\noindent Therefore, due to $s<p$, the following relation holds
$$\lim_{\|u\|\rightarrow \infty}J_\lambda(u)=+\infty,$$
for every $\lambda>0$.\par
\noindent Since $J_\lambda$ is coercive for every positive parameter $\lambda$, condition $(\textrm{i})$ is verified. Next, consider the function $u_\delta\in X$. Since
$$
\sum^{N}_{i=1}\frac{\partial^2 u_\delta(x)}{\partial x_i^2}=32d\left(\frac{2(N+2)l^2-3\tau(N+1)l+N\tau^2}{\tau^4}\right),
$$
for every $x\in B(x^0,\tau)\setminus B(x^0,\tau/2)$ and
$$
\sum^{N}_{i=1}\frac{\partial^2 u_\delta(x)}{\partial x_i^2}=0,\,\,\,\,\forall x\in (\bar\Omega\setminus B(x^0,\tau))\cup B(x^0,\tau/2),
$$
one has
\begin{equation}\label{phi}
\Phi(u_\delta)=\frac{\|u_\delta\|^p}{p}=\frac{2^{5p+1}\pi^{N/2}\delta^p}{\tau^{4p}\Gamma(N/2)p}\sigma_{p,N}({\tau}).
\end{equation}
\noindent Put
$$
r:=\frac{\gamma^p}{pk^p}.
$$
\noindent Now, it follows from $\delta>K_{p,N}(\tau)^{1/p}\gamma$ that $\Phi(u_\delta)>r$. We explicitly observe that, in view of (\ref{immersion}), one has
\begin{equation}\label{sublevels}
{\Phi^{-1}(]-\infty,r])}\subseteq\{u\in C^0(\bar\Omega):\ \|u\|_{\infty}\leq \gamma\}.
\end{equation}

\noindent  Moreover, taking (\ref{sublevels}) into account, a direct computation ensures that
\begin{equation}\label{cong1}
\displaystyle\sup_{u\in \overline{\Phi^{-1}(]-\infty,r[)}^{w}}\Psi(u)=\displaystyle\sup_{u\in {\Phi^{-1}(]-\infty,r])}}\Psi(u)\leq \displaystyle\int_{\Omega}\max_{|\xi|\leq \gamma}F(x,\xi)\;dx.
\end{equation}
\noindent At this point, by definition of $u_\delta$, we can clearly write
\begin{equation}\label{cong2}
\int_\Omega F(x,u_\delta(x))\;dx=R_{F}(\tau,\delta)
+\int_{B(x^0,\tau/2)}F(x,\delta)\;dx.
\end{equation}
\noindent By using hypothesis $(\textrm{h}_2)$, from \eqref{cong1} and \eqref{cong2}, we also have
$$
\displaystyle\sup_{u\in \overline{\Phi^{-1}(]-\infty,r[)}^{w}}\Psi(u)< \frac{r}{r+\Phi(u_\delta)}\Psi(u_\delta),
$$
taking into account that
$$
\frac{r}{r+\Phi(u_\delta)}=\frac{\tau^{4p}\Gamma(N/2)\gamma^p}{\tau^{4p}\Gamma(N/2)\gamma^p+k^p2^{5p+1}\pi^{N/2}\delta^{p}\sigma_{p,N}}=\eta(\gamma,\delta).
$$
So conditions $(\textrm{ii})$ and $(\textrm{iii})$ are verified by taking $\bar u:=u_\delta$. Thus, we can apply Theorem \ref{abstract} bearing in mind that
$$
\frac{\Phi(u_\delta)}{\Psi(u_\delta)-\displaystyle\sup_{u\in \overline{\Phi^{-1}(]-\infty,r[)}^{w}}\Psi(u)}\leq \lambda_1,
$$
and
$$
\frac{r}{\displaystyle\sup_{u\in \overline{\Phi^{-1}(]-\infty,r[)}^{w}}\Psi(u)}\geq \lambda_2.
$$
as well as
$$
\frac{hr}{\displaystyle r\frac{\Psi(u_\delta)}{\Phi(u_\delta)}-\displaystyle\sup_{u\in \overline{\Phi^{-1}(]-\infty,r[)}^{w}}\Psi(u)}\leq \lambda_{3,h}.
$$
 The proof is complete.
\end{proof}

\begin{remark}\label{semplice2}\rm{Assuming that
\begin{itemize}
\item [$(\textrm{j}_1)$] $\displaystyle F(x,\xi)\geq 0$ \textit{for every} $(x,\xi)\in \left(B(x^0,\tau) \setminus B(x^0,\tau/2)\right)\times [0,\delta]$$;$
\item [$(\textrm{j}_2)$] \textit{For every} $|\xi|\leq \gamma$ \textit{one has}
$$\displaystyle\int_{\Omega}\max_{|\xi|\leq \gamma}F(x,\xi)\;dx<\eta(\gamma,\delta)\int_{B(x^0,\tau/2)}F(x,\delta)\;dx,$$
\end{itemize}
it follows that hypothesis $(\textrm{h}_1)$ in Theorem \ref{Main} automatically hold.}
\end{remark}

\begin{remark}\label{general}
\rm{We point out that hypothesis $(\textrm{h}_2)$ in Theorem \ref{Main} can be stated in a more general form. Precisely, fix $x^0\in \Omega$ and pick $r_1,r_2\in\R$ with $r_2>r_1>0$, such that $B(x^0,r_1)\subset B(x_0,r_2)\subseteq \Omega$. Moreover, set\par
$$
\sigma_{p,N}(r_1,r_2):=\int_{r_1}^{r_2}|(N+2)s^2-(N+1)(r_1+r_2)s+Nr_1r_2|^ps^{N-1}ds,
$$
\noindent and denote
$$
K_{p,N}(r_1,r_2):=\frac{(r_2-r_1)^{3p}(r_1+r_2)^p\Gamma(N/2)}{2^{2p+1}3^p\pi^{N/2}k^p\sigma_{p,N}(r_1,r_2)}.
$$
At this point, let
$v_\delta$ the function be defined as follows,
\[
\displaystyle{
\small{
v_\delta(x):=\left\{
\begin{array}{ll}
0 & \mbox{ if $x \in \bar\Omega \setminus B(x^0,r_2)$} \\\\
\frac{\delta(3(l^4-r_2^4)-4(r_1+r_2)(l^3-r_2^3)+6r_1r_2(l^2-r_2^2))}{(r_2-r_1)^3(r_1+r_2)}
& \mbox{ if $x \in B(x^0,r_2) \setminus B(x^0,r_1)$} \\\\
\delta & \mbox{ if $x \in B(x^0,r_1),$}
\end{array}
\right.
}
}
\]
\noindent where $l:=\sqrt{\sum_{i=1}^{N}(x_i-x^0_i)^2}$.\par
 If $\delta$ and $\gamma$ in Theorem \ref{Main} satisfy $\delta>K(r_1,r_2)^{1/p}\gamma$, instead of $(\textrm{h}_1)$, hypothesis $(\textrm{h}_2)$ can be replaced by the following assumption, namely $(\textrm{h}_2^{\star})$:
\begin{equation*}
\displaystyle\int_{\Omega}\max_{|\xi|\leq \gamma}F(x,\xi)\;dx<\frac{r}{r+\Phi(v_\delta)}\left(R_{F}(r_1,r_2,\delta)
+\int_{B(x^0,r_1)}F(x,\delta)\;dx\right),
\end{equation*}
where
$$
r:=\frac{\gamma^p}{pk^p},
$$
$$
R_{F}(r_1,r_2,\delta):=\int_{B(x^0,r_2) \setminus B(x^0,r_1)}F(x,v_\delta(x))\;dx,
$$
and
$$
\Phi(v_\delta)=\frac{\delta^p}{pk^pK_{p,N}(r_1,r_2)}.
$$
Then for each
$$\lambda \in \Lambda_{1}^{\star}:=\left]\frac{\Phi(v_\delta)}{\Psi(v_\delta)-\displaystyle\int_{\Omega}\max_{|\xi|\leq \gamma}F(x,\xi)\;dx},\frac{\gamma^p}{\displaystyle pk^p\int_{\Omega}\max_{|\xi|\leq \gamma}F(x,\xi)\;dx}\right[,$$ the equation
\begin{equation}\label{equationx}
J_\lambda(u)=\Phi'(u)-\lambda\Psi'(u)=0,
\end{equation}
has at least three distinct solutions in $X$ and, moreover, for each $h>1$, there exists an open interval
$$
\Lambda_2^{\star}\subset\left[0,\frac{hr}{\displaystyle r\frac{\Psi(v_\delta)}{\Phi(v_\delta)}-\displaystyle\sup_{x\in \overline{\Phi^{-1}(]-\infty,r[)}^{w}}\Psi(x)}\right],
$$
and a positive real number $\sigma>0$ such that, for each $\lambda\in \Lambda_2$, the equation \eqref{equationx}
has at
least three solutions in $X$ whose norms are less than $\sigma$.
It is clear that if $r_1=\tau/2$ and $r_2=\tau$, condition $(\textrm{h}_2^{\star})$ coincides with $(\textrm{h}_2)$.\par
}
\end{remark}
Now, for completeness, we analyze the autonomous case
\begin{equation}\tag{$G_\lambda^{f}$} \label{4N}
\left\{
\begin{array}{ll}
\Delta(|\Delta u|^{p-2}\Delta u)=\lambda f(u)\quad & {\rm in\ }\Omega\\
u=\Delta u=0\quad & {\rm on\ }\partial\Omega,
\end{array}
\right.
\end{equation}
where $f:\erre\rightarrow \erre$ is a continuous function. With the above notations, let us define
$$
G_{F}(\tau,\delta):=\int_{B(x^0,\tau) \setminus B(x^0,\tau/2)}F(u_\delta(x))\;dx.
$$
\noindent Finally, the symbol $``\meas(B(x^0,\tau/2))"$ denotes the Lebesgue measure of the ball $B(x^0,\tau/2)$.
\begin{theorem}\label{Main2}
Let $f\in C^0(\R)$ and put
$$
F(\xi):=\int_0^{\xi} f(t)dt,\quad\forall\; \xi\;\R.
$$ Assume that there exist two positive constants $\gamma$ and $\delta$ such that condition $(\rm{h}_1)$ hold in addition to
\begin{itemize}
\item [$(\textrm{h}_2')$]
$\displaystyle F(\xi)<\frac{\eta(\gamma,\delta)}{\meas(\Omega)}\left(G_{F}(\tau,\delta)
+\meas(B(x^0,\tau/2))F(\delta)\right),$
for every $|\xi|\leq\gamma$.
\end{itemize}
Moreover, require that
\begin{itemize}
\item [$(\textrm{h}_3')$]
There exist two positive constants $b$ and $s$ with $s<p$ such that
$$
F(\xi)\leq b(1+|\xi|^s).
$$
\end{itemize}
Then, for each
$$\lambda \in \Lambda_{1}':=\left]\lambda_1',\lambda_2'\right[,$$

\noindent where
$$
\lambda_1':=\frac{2^{5p+1}\pi^{N/2}\sigma_{p,N}({\tau})\delta^p/\meas(\Omega)}{\tau^{4p}\Gamma(N/2)p\left(\displaystyle\frac{G_{F}(\tau,\delta)
+\displaystyle\meas(B(x^0,\tau/2))F(\delta)}{\meas(\Omega)}-\displaystyle\max_{|\xi|\leq \gamma}F(\xi)\right)},
$$
and
$$
\lambda_2':=\frac{\gamma^p}{\displaystyle pk^p\meas({\Omega})\max_{|\xi|\leq \gamma}F(\xi)},
$$
problem \eqref{4N}
has at least three distinct solutions in $X$ and, moreover, for each $h>1$, there exists an open interval
$$
\Lambda_2'\subset\left[0,\lambda_{3,h}'\right],
$$
where
$$
\lambda_{3,h}':=\frac{h\gamma^p/(p\meas({\Omega})k^p)}{\displaystyle\frac{\gamma^p\left(G_{F}(\tau,\delta)
+\displaystyle\meas(B(x^0,\tau/2))F(\delta)\right)\tau^{4p}\Gamma(N/2)}{\displaystyle 2^{5p+1}k^p\pi^{N/2}\sigma_{p,N}(\tau)\meas({\Omega})\delta^p}-\max_{|\xi|\leq \gamma}F(\xi)},
$$
and a positive real number $\sigma>0$ such that, for each $\lambda\in \Lambda_2'$, problem \eqref{4N}
has at
least three solutions in $X$ whose norms are less than $\sigma$.
\end{theorem}
\begin{remark}\label{semplice}\rm{The following two conditions
\begin{itemize}
\item [$(\textrm{j}_1')$] $\displaystyle G_{F}(\tau,\delta)\geq 0$$;$
\item [$(\textrm{j}_2')$] \textit{For every} $|\xi|\leq \gamma$ \textit{one has}
$$\displaystyle F(\xi)<\eta(\gamma,\delta)\frac{\meas(B(x^0,\tau/2))}{\meas(\Omega)}F(\delta),$$
\end{itemize}
imply hypotheses $(\textrm{h}_1')$ in Theorem \ref{Main2}.\par
 \noindent Furthermore, assumption $(\textrm{j}_1')$ is verified by requiring that $F(\xi)\geq 0$ for every $\xi\in[0,\delta]$. Moreover, if $f$ is nonnegative, hypothesis $(\textrm{j}_1')$ automatically holds and $(\textrm{j}_2')$ attains a more simply form
$$\displaystyle F(\gamma)<\eta(\gamma,\delta)\frac{\meas(B(x^0,\tau/2))}{\meas(\Omega)}F(\delta).$$
Hence, Theorem \ref{intro} of Introduction is a direct consequence of the above observations. Indeed, let $f$ be a nonnegative continuous function such that $f(t)=0$ for every $t\in [-\gamma,\gamma]$. Bearing in mind that $f$ is not identically zero, there exists $\delta>\gamma\max\left\{1,K_{p,N}(\tau)^{1/p}\right\},$ such that
$$\displaystyle 0=F(\gamma)<\eta(\gamma,\delta)\frac{\meas(B(x^0,\tau/2))}{\meas(\Omega)}F(\delta).$$
Finally, we observe that if
\begin{itemize}
\item [$(\textrm{h}_3^{\star})$]$
\displaystyle\lim_{|t|\rightarrow \infty
}\frac{f(t)}{|t|^{s-1}}=0,$
\end{itemize}
\noindent for some $1\leq s\leq p$, the functional $J_\lambda$ is coercive. We give just some computations in the case $s=p$; analogous conclusion holds for $s\in [1,p[$. So, fix $\lambda>0$ and pick $\varepsilon<1/(\lambda k^p\meas(\Omega))$. Now, by our assumption at infinity, there exists $c(\varepsilon)>0$ such that
$$
|f(t)|\leq \varepsilon |t|^{p-1}+c(\varepsilon),\,\,\,\,\forall\; t\in\R.
$$
Then the previous inequality gives
$$
\displaystyle F(\xi)\leq \frac{\varepsilon}{p}|\xi|^p+c(\varepsilon)|\xi|,\,\,\,\,\forall\; \xi\in\R,
$$
\noindent and, consequently, taking into account (\ref{immersion}), one has
$$
\Psi(u)\leq \left(\frac{\varepsilon k^p}{p}\|u\|^p+c(\varepsilon)k\|u\|\right)\meas(\Omega),\,\,\,\,\,\,\forall\; u\in X.
$$
\noindent Since, for every $u\in X$, the following inequality holds
$$
J_\lambda(u)\geq \left(\frac{1}{p}-\lambda\frac{\varepsilon k^p}{p}\meas(\Omega)\right)\|u\|^p-\lambda c(\varepsilon)k\|u\|\meas(\Omega),
$$
the functional $J_\lambda$ is coercive.\par
Thus, all the assumptions (with $(\textrm{h}_3^{\star})$ instead of $(\textrm{h}_3')$) of Theorem \ref{Main2} are verified and the conclusion follows. For completeness we also note that Theorem \ref{intro} of Introduction is still true without sign assumption on $f$ on the half-line $]-\infty,\gamma[$.}
\end{remark}

At the end we exhibit a concrete application of our results.

\begin{example}\label{esempio}
{\rm
Let $\Omega$ be a nonempty bounded open subset of the Euclidean space $\R^{3}$ with a smooth boundary $\partial\Omega$ and define $f:\R\rightarrow \R$ as follows
\[
f(t):= \left\{
\begin{array}{ll}
\displaystyle 0 & \mbox{ if\, $t<2$}\\\\
\displaystyle \sqrt{t-2} & \mbox{ if $t\geq 2$},
\end{array}
\right.
\]
whose potential is given by
\[
F(\xi):= \left\{
\begin{array}{ll}
\displaystyle 0 & \mbox{ if\, $\xi<2$}\\\\
\displaystyle {\frac{2(\xi-2)^{3/2}}{3}} & \mbox{ if $\xi\geq 2$}.
\end{array}
\right.
\]
Consider
the following problem
 \begin{equation}\tag{$H_\lambda^{f}$} \label{N2}
\left\{
\begin{array}{ll}
\Delta^2 u=\lambda f(u)\quad & {\rm \ }\Omega\\
u=\Delta u=0\quad & {\rm \ }\partial\Omega.
\end{array}
\right.
\end{equation}

\noindent Arguing as in Remark \ref{semplice} we can observe that there exist two positive constants $\gamma=2$ and
 $$
 \delta>2\left\{1,K_{2,3}(\tau)^{1/p}\right\},
 $$
 such that, taking into account Remark \ref{semplice}, all the conditions of Theorem \ref{Main2} hold. Then, for each
$$\lambda \in \Lambda_{1}':=\left]\lambda_1^{\star},+\infty\right[,$$
\noindent where
$$
\lambda_1^{\star}:=\frac{2^{10}\pi^{3/2}\sigma_{2,3}({\tau}){\delta}^2}{\tau^{8}\Gamma(3/2)\left(G_{F}(\tau,\delta)
+\displaystyle\meas(B(x^0,\tau/2))F(\delta)\right)},
$$
\noindent problem \eqref{N2}
has at least three distinct (two nontrivial) solutions in $W^{2,2}(\Omega)\cap W_0^{1,2}(\Omega)$ and, moreover, for each $h>1$, there exists an open interval
$$
\Lambda_2'\subset\left[0,\lambda_{3,h}^{\star}\right],
$$
where
$$
\lambda_{3,h}^{\star}:=\frac{2^{10}\pi^{3/2}\sigma_{2,3}(\tau)\delta^2h}{\displaystyle {\tau^{8}\Gamma(3/2)\left(G_{F}(\tau,\delta)
+\displaystyle\meas(B(x^0,\tau/2))F(\delta)\right)}}=h\lambda_1^{\star},
$$
and a positive real number $\sigma>0$ such that, for each $\lambda\in \Lambda_2'$, problem \eqref{N2}
has at
least three (two nontrivial) solutions in $W^{2,2}(\Omega)\cap W_0^{1,2}(\Omega)$ whose norms are less than $\sigma$.
}
 \end{example}

 \medskip
 \indent {\bf Acknowledgements.}  This paper was written when the first author was a visiting professor at the University of Ljubljana in 2012. He expresses his gratitude  for the warm hospitality.
 The research was supported in part by the SRA grants P1-0292-0101 and J1-4144-0101.


\begin{thebibliography}{99}
\bibitem{Bonanno}
\textsc{G. Bonanno}, {\em A critical points theorem and nonlinear differential problems}, J. Global
Optim. \textbf{28} (2004), 249-258.
%
\bibitem{BoMo}
\textsc{G. Bonanno and G. Molica Bisci}, \textit{A remark on a perturbed Neumann problem}, Stud. Univ. Babes-Bolyai Math. LV, \textbf{4} (2010), 17-25.
%
%
\bibitem{BoMo2}
\textsc{G. Bonanno and G. Molica Bisci}, \textit{Infinitely many solutions for a Dirichlet problem involving the
$p$-Laplacian}, Proc. Roy. Soc. Edinburgh Sect. A \textbf{140} (2010), 737-752.
%
%
\bibitem{CL}
\textsc{P. Candito and R. Livrea}, {\em Infinitely many solutions for
a nonlinear Navier boundary value problem involving the $p$-biharmonic}, Stud. Univ. "Babe\c{s}–-Bolyai" Math. \textbf{50} (2010), 41-51.
%
\bibitem{CM}
\textsc{P. Candito and G. Molica Bisci}, {\em Multiple solutions for a Navier boundary value problem involving the $p$-biharmonic}, Discrete and Continuous Dynamical Dystems Series S {\bf 5} (4) (2012), 741-751.
%
\bibitem{GG}
\textsc{H. M. Guo and D. Geng}, {\em Infinitely many solutions for
the Dirichlet problem involving the $p$-biharmonic like
equation}, J. South China Normal Univ. Natur. Sci. Ed. {\bf 28}
(2009), 18-21.
%
\bibitem{GST}
\textsc{M.R. Grossinho, L. Sanchez and S.A. Tersian}, {\em On the solvability of a boundary value problem for a fourth-order ordinary differential equation},
Appl. Math. Lett. \textbf{18} (2005), 439-444.
%
\bibitem{HZ}
\textsc{G. Han and Z. Xu}, {\em Multiple solutions of some nonlinear fourth-order beam equations}, Nonlinear Anal. \textbf{68} (2008), 3646-3656.
%
\bibitem{KRV} \textsc{A. Krist\'{a}ly, V. R\u{a}dulescu and Cs. Varga}, {\it Variational Principles in Mathematical Physics, Geometry, and Economics:
Qualitative Analysis of Nonlinear Equations and Unilateral Problems}, Encyclopedia of Mathematics and its Applications, No.~136, Cambridge University Press, Cambridge, 2010.
%
\bibitem{LT}
\textsc{C. Li and C-L. Tang}, {\em Three solutions for a Navier
boundary value problem involving the $p$-biharmonic}, Nonlinear Anal. {\bf 72} (2010), 1339-1347.
    %
    \bibitem{LT2}
\textsc{C. Li and C-L. Tang}, {\em Existence of three solutions for $(p,q)$-biharmonic systems,} Nonlinear Anal. {\bf 73} (2010), 796-805.
%
\bibitem{Gr} \textsc{J.L. Lions}, {\em Quelques M\'{e}thodes de R\'{e}solution des Probl\`{e}mes aux Limites Non Lineaires,} Dunod, Paris,
France, 1969.
    %
    \bibitem{LL}
    \textsc{X.-L. Liu and W.-T. Li}, {\em Existence and multiplicity of solutions for fourth-order boundary values problems with parameters}, J. Math. Anal. Appl. \textbf{327} (2007), 362-375.
%
\bibitem{LM}
\textsc{A.C. Lazer, P.J. McKenna}, {\em Large-amplitude periodic oscillations in suspension bridges: Some new connections with nonlinear analysis}, SIAM Rev. {\bf 32} (1990), 537-578.
%
\bibitem{MP}
\textsc{A.M. Micheletti and A. Pistoia}, {\em Multiplicity results
for a fourth-order semilinear elliptic problems}, Nonlinear Anal.
{\bf 60} (1998), 895-908.
%
\bibitem{Ric}
\textsc{B. Ricceri}, {\em A general variational principle and some of its
applications}, J. Comput. Appl. Math. {\bf 133} (2000), 401-410.
%
\bibitem{Talenti}
\textsc{G. Talenti}, {\em Elliptic equations and rearrangements}, Ann. Scuola Norm. Sup. Pisa Cl.
Sci.  {\bf 3} (4) (1976), 697-718.
%
\bibitem{WS}
\textsc{Y. Wang and Y. Shen}, {\em Infinitely many sign-changing
solutions for a class of biharmonic equations without symmetry},
Nonlinear Anal. {\bf 71} (2009), 967-977.
%
\bibitem{WZ}
\textsc{W. Wang and P. Zhao}, {\em Nonuniformly nonlinear elliptic
equations of $p$--biharmonic type}, J. Math. Anal. Appl. {\bf 348}
(2008), 730-738.
%
\bibitem{YGY}
\textsc{Z. Yang, D. Geng and H. Yan}, {\em Existence of multiple
solutions for a semilinear biharmonic equations with critical
exponent}, Acta Math. Sci. Ser. A {\bf 27} (2006), 129-142.
%
\end{thebibliography}
\end{document}